\documentclass[11pt]{article}

\usepackage{amsmath, mathrsfs, amsthm, amssymb, stmaryrd, cancel, hyperref, relsize}
\usepackage{graphicx}
\usepackage{xfrac}
\hypersetup{pdfstartview={XYZ null null 1.25}}
\usepackage[all]{xy}
\usepackage[normalem]{ulem}

\theoremstyle{plain}
\newtheorem{Thm}{Theorem}[section]
\newtheorem{Lem}[Thm]{Lemma}
\newtheorem{Cor}[Thm]{Corollary}
\newtheorem{Prop}[Thm]{Proposition}

\newtheorem{Not}[Thm]{Notation}

\theoremstyle{definition}
\newtheorem{Def}[Thm]{Definition}
\newtheorem{Ex}[Thm]{Example}

\newtheorem{Prob}[Thm]{Problem}

\makeatletter
\newcommand{\labitem}[2]{%
\def\@itemlabel{\textbf{#1}}
\item
\def\@currentlabel{#1}\label{#2}}
\makeatother

\DeclareMathOperator{\A}{\mathcal A}
\DeclareMathOperator{\id}{\mathbf{id}}
\DeclareMathOperator{\ide}{\mathbf{id}}
\DeclareMathOperator{\G}{\Gamma_{\scriptscriptstyle{D}}}
\DeclareMathOperator{\dom}{\mathbf{dom}}
\DeclareMathOperator{\ran}{\mathbf{ran}}
\DeclareMathOperator{\comp}{\mathbin{;}}
\def\conv{\smile}
\DeclareMathOperator{\au}{\mathit{a}^\uparrow}
\DeclareMathOperator{\Pc}{\mathit{P}^*}

\DeclareMathOperator{\mP}{\mathcal{P}}
\DeclareMathOperator{\hf}{\hat{\mathit{f}}}
\DeclareMathOperator{\mPs}{\mathcal{P}^*}
\DeclareMathOperator{\Pcd}{\mathit{P}^{*\delta}}
\DeclareMathOperator{\pu}{\mathit{p}^\uparrow}
\DeclareMathOperator{\bw}{\bigwedge}

\newcommand{\Qs}{\bcancel{\mathit{Q}}^{\mathit n}}
\newcommand{\N}{\mathbb{N}}
\newcommand{\hfp}{\hat{\mathit{f}}^{\scriptscriptstyle+}}

\newcommand{\compC}{\comp_{\scriptscriptstyle \gamma}}
\newcommand{\domC}{\dom_{\scriptscriptstyle \gamma}}
\newcommand{\ranC}{\ran_{\scriptscriptstyle \gamma}}
\newcommand{\convC}{\conv_{\scriptscriptstyle \gamma}}
\newcommand{\gammap}{\gamma^{\scriptscriptstyle+}}
\newcommand{\idC}{\id_{\scriptscriptstyle \gamma}}
\newcommand{\oC}{0_{\scriptscriptstyle \gamma}}

\def\set#1{{ \{ #1\} }}

\def\c#1{{\mathcal #1}}

\title{Meet-completions and ordered domain algebras}

\author{R Egrot and R Hirsch}

\begin{document}
\maketitle
\begin{abstract}
Using the well-known equivalence between meet-completions of posets and standard closure operators we show a general method for constructing meet-completions for isotone poset expansions. With this method we find a meet-completion for ordered domain algebras which simultaneously serves as the base of a representation for such algebras, thereby proving that ordered domain algebras have the finite representation property. We show that many of the equations defining ordered domain algebras are preserved in this completion but associativity, ({\bf D2}) and ({\bf D6}) can fail.
\let\thefootnote\relax\footnotetext{Corresponding author: Robin Hirsch, Department of Computer Science, University College,
Gower Street,
London WC1E 6BT,
email: r.hirsch@ucl.ac.uk
}
\end{abstract}

\begin{section}{Introduction}
When considering algebras of binary relations, it is generally not the case that a finite representable algebra has a representation on a finite base.  Indeed, in any signature which includes the identity,  intersection and composition operators, any representation of the \emph{Point Algebra} \cite{VilKau86} interprets each diversity atom as a dense linear order and so the representation is necessarily infinite.  

There are two well-known cases where we do have the finite representation property.  For the signature with identity, converse and composition only, the Cayley representation maps an algebra element $a$ to the binary relation $\set{(x, x;a):x\in\c A}$ over the algebra itself.  At the other extreme, for the signature consisting of Boolean operators only we may modify the standard Stone representation (which represents elements as unary relations) and represent an element $a$ as the identity relation over the ultrafilters containing $a$. Similarly, for a signature with solely an order relation $\leq$ (in other words a poset, $P$ say) we may construct a representation whose base $P^*$ consists of the upward closed subsets of $P$. An element $p$ of $P$ is represented as the identity restricted to $\hat p=\set{u\in P^*:p\in u}$. Clearly $P^*$ is finite if $P$ is. Note that as well as providing the base of a representation, $(P^*, \cup, \cap)$ forms a complete distributive lattice and is a completion of $P$ (see section \ref{S;McompClos}).

An interesting case, then, is where the signature includes both composition and an order relation.  For the signature with composition, converse, the domain operator and an ordering (Ordered Domain Algebra)   a construction in \cite{HirMik13} had  aspects of the Cayley representation but also aspects of the upward closed set representation for a poset.  Each element of an  ordered domain algebra (ODA) is represented as a set of pairs  of upward closed subsets of the algebra, but in order to make the representation work for the non-Boolean operators, these upward closed subsets are required to have certain other closure properties.  \cite{Bre77} had already provided a complete, finite set of axioms defining the class of representable ordered domain algebras, the construction in \cite{HirMik13} also showed that a finite representable ODA has a representation on a finite base.  Here we show, further, that the operators may be lifted from an ordered domain algebra $\c A$ to an algebra $\G[\c A]$  whose universe is the set of all closed subsets of $\c A$, where the ordering $\leq$ is reverse inclusion $\supseteq$ and where the other operators are lifted from $\c A$.  We show that $\G[\c A]$ is a completion of $\c A$ and it obeys many  of the equations defining ODAs.  On the other hand, rather  important properties, like the associativity of composition, are shown to be fallible in $\G[\c A]$.

Our main results are the following. We show how to lift the isotone operators of a poset expansion to operators on a meet-completion of the poset in definition~\ref{def:bullet}.   We prove that certain inequalities are preserved when passing to this completion (given certain conditions on the inequalities and the completion) in corollary~\ref{C;sahl}.  Focusing on ODAs, we define a particular completion $\G$ in definitions~\ref{D;closed}, \ref{def:closed2} and restate the result that this completion can act as the base of a representation of an ODA in theorem~\ref{thm:rep}.  In proposition~\ref{P;AxHold} we show that all but three of the equations defining ODAs are preserved in this completion, and in examples~\ref{E;2fail}, \ref{E;6fail} and \ref{E;Afail} we show that those three equations can fail in the completion.

The  remainder of this paper is divided as follows. In the next section we give the basic definitions for meet-completions and standard closure operators and provide a proof of the well-known correspondence between them.  In section \ref{S;InPreserve} we explain how to extend isotone operators on a poset to a meet-completion of that poset. This provides a method for extending poset meet-completions to meet-completions of isotone poset expansions. We investigate some general rules governing the preservation of inequalities by meet-completions of isotone poset expansions using this method.  In section \ref{S;reps} we define ordered domain algebras, and in section \ref{S;ODAcomp} we apply the considerations of sections \ref{S;McompClos} and \ref{S;InPreserve} to construct a completion for ODAs and determine which ODA equations it preserves. We show how this completion can be used as the base of a representation for that algebra.  Finally in section \ref{S:conclusions} we draw some conclusions from these results and offer suggestions for further work in this area.

\end{section}

\begin{section}{Meet-completions and closure operators}\label{S;McompClos}
The material in this section is well-known, dating back to the pioneering work of Ore \cite{Ore43a,Ore43,Ore44}. The aim here is to provide formulations best suited for the work we undertake in later sections.

  Throughout, for any unary function $f$ and subset $S$ of the domain of $f$ we write $f[S]$ for $\set{f(s):s\in S}$, more generally for an $n$-ary function $g$ (written prefix)  we may apply $g$ pointwise to a sequence of   subsets $S_i$ of the domain of $g$ ($i=1,\ldots,n$) and write $g[S_1\times \ldots\times S_{n}]$ to denote $\set{g(s_1,\ldots, s_n): s_i\in S_i\;\mbox{ for } 1\leq i\leq n}$.  Two exceptions to this notational convention, where the functions are not written prefix, are the unary operator ${}^\conv$ and   the binary $;$ and we simply write $S^\conv,\; S;T$ for $\set{s^\conv:s\in S},\; \set{s;t:s\in S,\; t\in T}$ respectively.   For any subset $S$ of a poset $P$ we write $S^\uparrow$ for $\set{p\in P: \exists s\in S,\; s\leq p}$. For $p\in P$ we write $\pu$ as shorthand for $\{p\}^\uparrow$.

\begin{Def}[Completion]\label{def:completion}
Given a poset $P$ we define a completion of $P$ to be a complete lattice $Q$ and an order embedding $e\colon P\to Q$.
\end{Def}

We say $e\colon P\to Q$ is a \emph{meet-completion} when $e[P]$ is \emph{meet-dense} in $Q$. That is, when $q=\bw\{e(p):p\in P \text{ and } e(p)\geq q\}$ for all $q\in Q$.

\begin{Def}[$\Pc$]\label{D;Pstar}
If $P$ is a poset define $\Pc$ to be the complete lattice of up-sets (including $\emptyset$) of $P$ ordered by reverse inclusion ($S_1\leq S_2 \iff S_1\supseteq S_2$). The order dual $\Pcd$ is the lattice of up-sets ordered by inclusion with bottom element $\emptyset$.
\end{Def}

It's easy to see that the map $\iota\colon P\to \Pc$ defined by $\iota(p)=\pu$ defines a meet-completion of $P$ (note though that $\iota$ will not map the top element of $P$ (if it exists) to the top element of $\Pc$, as the top element of $\Pc$ will be $\emptyset$). This particular completion plays an important role in the theory of meet-completions.

\begin{Def}[Closure operator]\label{D;closop}
Given a poset $P$, a closure operator on $P$ is a map $\Gamma\colon P\to P$ such that
\begin{enumerate}
\item $p\leq\Gamma(p)$ for all $p\in P$,
\item $p\leq q\implies \Gamma(p)\leq \Gamma(q)$ for all $p,q\in P$, and
\item $\Gamma(\Gamma(p))=\Gamma(p)$ for all $p\in P$.
\end{enumerate}
\end{Def}

Following \cite{Ern83} we say a closure operator $\Gamma$ on $\Pc$ or $\Pc^\delta$  is \emph{standard} when  $\Gamma(\pu)=\pu$ for all $p\in P$.

It is well-known that a meet-completion $e\colon P \to Q$ defines a standard closure operator $\Gamma_e\colon \Pcd\to\Pcd$ by $\Gamma_e(S)=\{p\in P:e(p)\geq \bw e[S]\}$ (we take the dual of $\Pc$ as otherwise condition 1 of Definition \ref{D;closop} fails). In this case $Q$ is isomorphic to the lattice $\Gamma_e[\Pc]$ of $\Gamma_e$-closed subsets of $\Pc$. This isomorphism is given by the map $h_e\colon Q\to\Gamma_e[\Pc]$ defined by $h_e(q)=\{p\in P:e(p)\geq q\}$. Note that we are purposefully taking $\Pc$ rather than $\Pc^\delta$ here as we want to order by reverse inclusion. This is technically an abuse of notation as $\Gamma_e$ is originally defined on $\Pc^\delta$, but as these structures have the same carrier hopefully our meaning is clear.  

Conversely, whenever $\Gamma$ is a standard closure operation on $\Pcd$ it induces a meet-completion $e_{\Gamma}\colon P\to \Gamma[\Pc]$ defined by $e_\Gamma(p)=\pu$. For $S\in\Pc$ we have $\Gamma_{e_\Gamma}(S)=\{p\in P:\pu\geq\bw \{\pu:p\in S\}\}=\{p:\pu\subseteq \Gamma(S)\}=\Gamma(S)$, so $\Gamma_{e_\Gamma}=\Gamma$. Moreover, for all $p\in P$ we have $e_{\Gamma_e}(p)=\pu=h_e\circ e(p)$ so the diagram in figure \ref{F;com} commutes.

For convenience we summarize the preceding discussion below:

\begin{Not}\label{N;con}
For meet-completion $e:P\to Q$, and standard closure operator $\Gamma\colon \Pcd\to\Pcd$ we define:
\begin{itemize}
\item[] $\begin{aligned}\mathbf{\Gamma_e}: \Pcd&\to\Pcd\\
S&\mapsto \{p\in P:e(p)\geq \bw e[S]\}
\end{aligned}$
\item[] $\begin{aligned} \mathbf{h_e}:Q&\to\Gamma_e[\Pc]\\
q&\mapsto \{p\in P:e(p)\geq q\}\end{aligned}$
\item[] $\begin{aligned} \mathbf{h_e^{-1}}:\Gamma_e[\Pc]&\to Q\\
S&\mapsto \bw e[S]\end{aligned}$
\item[] $\begin{aligned}\mathbf{e_\Gamma}:P &\to \Gamma[\Pc]\\
p&\mapsto \pu
\end{aligned}$
\end{itemize}
\end{Not}

We have seen,  
with notation from \ref{N;con}:  
\begin{Lem}
For any meet completion $e:P\rightarrow Q$ and any standard closure operator $\Gamma\colon \Pcd\to\Pcd$
\begin{enumerate}
\item $\Gamma_{e_\Gamma}=\Gamma$
\item $e_{\Gamma_e} = h_e\circ e$
\end{enumerate}
\end{Lem}

\begin{figure}
\[\xymatrix{P\ar[r]^e\ar[d]_{e_{\Gamma_e}} & Q\ar@{->}[dl]^{h_e}\\
\Gamma_e[\Pc]}\]
\caption{The equivalence between meet-completions and standard closure operators.}
\label{F;com}
\end{figure}

\begin{Lem}\label{L;isomunique}
If $e_1\colon P\to Q_1$ and $e_2\colon P \to Q_2$ are meet-completions of $P$ and $g\colon Q_1\to Q_2$ is an isomorphism such that $g\circ e_1= e_2$, then $g$ is unique with this property. 
\end{Lem}
\begin{proof}
Suppose $h$ is another such isomorphism. Then for all $p\in P$, and for all $q\in Q_1$, we have 
\begin{align*}
 e_2(p)\geq g(q) &\iff g\circ e_1(p)\geq g(q)\iff e_1(p)\geq q \\
&\iff h\circ e_1(p)\geq h(q)\iff e_2(p)\geq h(q),
\end{align*}
 so $\{p\in P:e_2(p)\geq g(q)\}=\{p\in P:e_2(p)\geq h(q)\}$ and thus by meet-density we are done. 
\end{proof}

\begin{Thm}\label{T;charac}
If $e\colon P\to Q$ is a meet-completion then there is a unique isomorphism $h_e$ between $Q$ and $\Gamma_e[\Pc]$ such that the diagram in figure \ref{F;com} commutes.

Moreover, if $e_1\colon P\to Q_1$ and $e_2\colon P\to Q_2$ are meet-completions such that there is an isomorphism $h\colon Q_1\to Q_2$ with $h\circ e_1=e_2$ then $\Gamma_{e_1}=\Gamma_{e_2}$.
\end{Thm}
\begin{proof}
The isomorphism required has been given as $h_e: q \mapsto \{p\in P:e(p)\geq q\}$. That this is an isomorphism is easy to show using the fact that $e:P\to Q$ is a meet-completion. Uniqueness follows from Lemma \ref{L;isomunique}. 

Finally, if $h\colon Q_1\to Q_2$ with $h\circ e_1=e_2$ then 
\begin{align*}
\Gamma_{e_2}(S)&=\{p\in P:e_2(p)\geq \bw e_2[S]\}\\
&=\{p\in P:h\circ e_1(p)\geq \bw h\circ e_1[S]\}\\
&=\{p\in P:h\circ e_1(p)\geq h(\bw e_1[S])\}\\
&=\{p\in P:e_1(p)\geq \bw e_1[S]\}\\
&=\Gamma_{e_1}(S)
\end{align*}
\end{proof}

\end{section}

\begin{section}{Meet-completions and Cartesian products}
\begin{Def}[$e^n$]\label{L;e2}
If $P$ is a poset then given a map $e\colon P\to Q$ we can define a map $e^n\colon P^n\to Q^n$ by \begin{equation*}e^n((p_1,...,p_n))=(e(p_1),...,e(p_n)).\end{equation*}
\end{Def}

\begin{Lem}\label{L;topComp}
If $e\colon P\to Q$ is a meet-completion and $n\geq 2$ then $e^n\colon P^n\to Q^n$ will be a meet-completion if and only if $P$ has a top element $\top$ and $e(\top)$ is the top element of $Q$. 
\end{Lem}
\begin{proof}
Suppose first that $P$ has top $\top$ and $e(\top)$ is the top element of $Q$. Since a finite product of complete lattices is again a complete lattice it remains only to check that $e^n[P^n]$ is meet-dense in $Q^n$. Given $(q_1,...,q_n)\in Q^n$ we claim that $(q_1,...,q_n)=\bw\{e^n((p_1,...,p_n)):e(p_i)\geq q_i$ for all $i\in\{1,...,n\}\}$. Since $e(\top)$ is the top element of $Q$ we can be sure that  this infimum is of a non-empty set. Now, $(q_1,...,q_n)$ is clearly a lower bound, so suppose $(q'_1,...,q'_n)$ is another such lower bound. Then, for $i\in\{1,...,n\}$, we have $q'_i\leq e(p_i)$ for all $p_i\in P$ with $q_i\leq e(p_i)$, so by meet-density of $e[P]$ in $Q$ we have $q'_i\leq q_i$, and so $(q'_1,...,q'_n)\leq (q_1,...,q_n)$ as required.

For the converse note that $Q$ has a top element $\top_Q$ since it is complete, and if $e(p)<\top_Q$ for all $p\in P$ then $Q^n$ will contain elements of form $(e(p_1),\top_Q,e(p_3),...,e(p_n))$ which are not the infimum of any subset of $e[P]$. Example \ref{E;notComp} illustrates this issue.
\end{proof}

\begin{Ex}\label{E;notComp}
Let $P = \{p\}$, let $Q = \{q,\top\}$ with $q<\top$ and let $e(p)=q$. Then $e\colon P\to Q$ is represented in the following diagram.
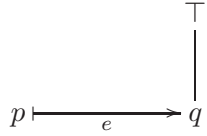
\[\xymatrix{ && \top\ar@{-}[d]\\
p\ar@{|->}[rr]_e && q
}\]
Now, $P^2=P$ and $e^2\colon P^2\to Q^2$ is as follows:
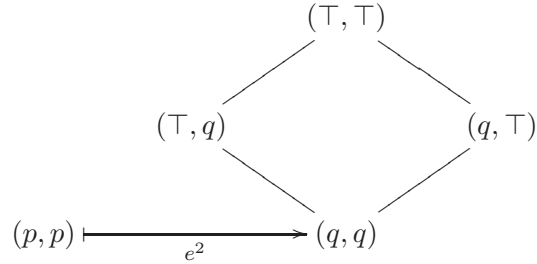
\[\xymatrix{ && (\top,\top)\ar@{-}[dr]\ar@{-}[dl]\\
&(\top,q)&&(q,\top)\\
(p,p)\ar@{|->}[rr]_{e^2} && (q,q)\ar@{-}[ur]\ar@{-}[ul]
}\]
Clearly $e^2\colon P^2\to Q^2$ is not a meet-completion.
\end{Ex}

Although $e^n\colon P^n\to Q^n$ is not, in general, a meet-completion we note that the problem elements cannot be below any member of $e^n[P^n]$, and so a true meet-completion can be obtained by simply removing them. We make a formal definition below: 

\begin{Def}[$\Qs$]
If $\top_{Q^n}$ is the top element of $Q^n$ we obtain $\Qs$ from $Q^n$ by removing all elements $(q_1,...,q_n)$ of $Q^n$ such that $\{(p_1,...,p_n)\in P^n : (e(p_1),...,e(p_n))\geq (q_1,...,q_n)\}=\emptyset$ and $(q_1,...,q_n)< \top_{Q^n}$. The ordering of the remaining elements is left unchanged. 
\end{Def}

\begin{Lem}
If $e\colon P\to Q$ is a meet-completion then the map $e^n\colon P^n \to \Qs$ is a meet-completion ($e^n$ is defined as in definition \ref{L;e2}).
\end{Lem}
\begin{proof}
$\Qs$ remains a complete lattice as aside from $\top_{Q^n}$ every element of $\Qs$ is below an element of $e^n[P^n]$. This also means that unless $(q_1,...,q_n)=\top_{Q^n}$ we have $\{e^n((p_1,...,p_n)):e(p_i)\geq q_i$ for all $i\in\{1,...,n\}\}\neq\emptyset$, and so we can adapt the first part of the proof of lemma \ref{L;topComp} to get the result. 
\end{proof}

Note that $\Qs$ embeds into $Q$ as a subalgebra.

\begin{Lem}
Given meet-completion $e\colon P \to Q$ and $n\geq 2$, the map $e^n\colon P^n\to Q^n$ is a meet-completion if and only if $Q^n=\Qs$.
\end{Lem}
\begin{proof}
$Q^n=\Qs$ if and only if $P$ has a top element $\top$ and $e(\top)$ is the top element of $Q$, so the result follows from lemma \ref{L;topComp}.
\end{proof}

\end{section}

\begin{section}{Meet-completions of isotone poset expansions}
\begin{Def}[\emph{poset expansion}]
A poset expansion is a structure $\mathcal P=(P, \leq, f_i:i\in I)$ such that $(P,\leq)$ is a poset and for each $i\in I$ there is $n_i\in \N$ with $f_i\colon P^{n_i}\to P$. Note that we define $0$ to be an element of $\N$. We say a poset expansion is \emph{isotone} when $f_i$ is isotone for all $i\in I$.
\end{Def}

If $e_1\colon P_1\to Q_1$ is a meet-completion, $e_2\colon P_2\to Q_2$ is any completion of $P_2$, and $f\colon P_1\to P_2$ is an isotone map, there is an intuitive method (introduced in \cite{Monk70}) for lifting $f$ to an isotone map $\hf\colon Q_1\to Q_2$, given by 
\begin{equation}\hf(q)=\bw\{e_2\circ f(p):e_1(p)\geq q\}\label{Eq;fhat}
\end{equation}
This lifting is illustrated by figure \ref{F;opliftEx} below.

\begin{figure}[ht]
\[\xymatrix{P_1\ar[d]_f\ar[r]^{e_1} & Q_1\ar[d]^{\hf}\\
P_2\ar[r]_{e_2}& Q_2 
}\]\caption{\label{F;opliftEx}An intuitive lift for an isotone map. }
\end{figure}

Applying this to the special case of the meet-completions $e\colon P \to Q$ and $e^n\colon P^n\to \Qs$, and an isotone map $f\colon P^n\to P$ we obtain:  
\begin{align*}\hf\colon \Qs&\to Q \\
q &\mapsto \bw\{e(f(p_1,..,p_n) : e^n(p_1,...,p_n)\geq q\}\end{align*}

The corresponding commuting square is shown in figure \ref{F;opliftEx2}.

\begin{figure}[ht]
\[\xymatrix{P^n\ar[d]_f\ar[r]^{e^n} & \Qs\ar[d]^{\hf}\\
P\ar[r]_{e}& Q 
}\]\caption{Lifting isotone operations to $\Qs$.}
\label{F;opliftEx2}
\end{figure}

We can extend $\hf$ to an order preserving map $\hfp\colon Q^n\to Q$ by defining:
\begin{align*}
\hfp \colon Q^n&\to Q\\
q&\mapsto \begin{cases} \hf(q) \text{   when }q\in\Qs\\ \top_{Q} \text{    otherwise}\end{cases}
\end{align*} 
Note that when $Q^n\neq \Qs$ we have $\hf(\top_{Q^n})=\hf(\top_{\Qs})=\top_{Q}$.
The situation can be summarized by the commuting diagram in figure \ref{F;oplift2} (here $\iota$ stands in both cases for the appropriate inclusion function). In this diagram the maps $\gamma$ and $\gammap$ are induced by the other maps. Lemma \ref{L;mapDef} gives an explicit definition for each.

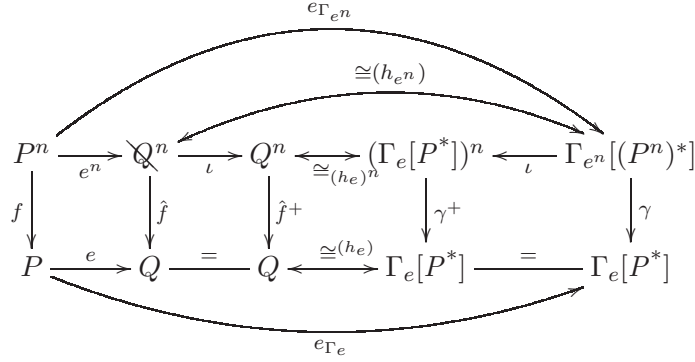
\begin{figure}[ht]
\[\xymatrix{P^n\ar@/^4pc/[rrrr]^{e_{\Gamma_{e^n}}}\ar[r]_{e^n}\ar[d]_f 
& \Qs\ar[d]^{\hf}\ar@{<->}@/^2pc/[rrr]^{\cong(h_{e^n})}\ar[r]_{\iota}
& Q^n \ar@{<->}[r]_{\cong_{(h_e)^n}} \ar[d]^{\hfp}
& (\Gamma_{e}[\Pc])^n\ar[d]^{\gammap}
&  \Gamma_{e^n}[(P^n)^*]\ar[d]^{\gamma}\ar[l]^\iota\\
P\ar[r]^{e}\ar@/_2pc/[rrrr]_{e_{\Gamma_e}} 
& Q\ar@{-}[r]^= 
& Q\ar@{<->}[r]^{\cong^{(h_e)}}
& \Gamma_{e}[\Pc]\ar@{-}[r]^=
& \Gamma_{e}[\Pc]
}\]\caption{\label{F;oplift2}Lifting isotone operations}
\end{figure}

\begin{Lem}\label{L;mapDef}
The maps $\gamma\colon \Gamma_{e^n}[(P^n)^*] \to \Gamma_{e}[\Pc]$ and $\gammap\colon (\Gamma_{e}[\Pc])^n\to \Gamma_{e}[\Pc]$ can be defined as follows:
\begin{enumerate}
\item $\gammap(C_1,...,C_n) = \begin{cases}\Gamma_e(f[C_1\times\ldots\times C_n]^\uparrow) \text{     when } C_i\neq\emptyset \text{ for all } i\in\{1,...,n\} \\ \Gamma_e(\emptyset) \text{     otherwise}\end{cases}$
\item $\gamma(C) = \begin{cases}\Gamma_e(f[C_1\times\ldots\times C_n]^\uparrow) \text{     when } C=(C_1,...,C_n)\neq\emptyset \\ \Gamma_e(\emptyset) \text{     otherwise}\end{cases}$
\end{enumerate}
\begin{proof}
First we check that the notation makes sense. Elements of $(P^n)^*$ are the up-closed subsets of $P^n$, so are either empty or of form $S_1\times\ldots\times S_n$ where the $S_i$ are non-empty up-closed subsets of $P$. Elements of $(\Pc)^n$ are of form $(S_1,...,S_n)$ where the $S_i$ are up-closed subsets of $P$ which may be empty. 
\\

The functions $\gamma$ and $\gammap$ are defined by composing maps from figure \ref{F;oplift2}. The relevant maps are as follows:
\begin{itemize}
\item[$(h_e)^n\colon$] $\begin{aligned}\\  Q^n &\to (\Gamma_e[P^*])^n\\
(q_1,...,q_n) &\mapsto (\{p\in P : e(p)\geq q_1\},...,\{p\in P : e(p)\geq q_n\})
\end{aligned}$

\item[$((h_e)^n)^{-1}\colon$] $\begin{aligned}\\  (\Gamma_e[P^*])^n &\to Q^n\\
(C_1,...,C_n)&\mapsto (\bw e[C_1],...,\bw e[C_n])
\end{aligned}$

\item[$\hfp\colon$] $\begin{aligned}\\ \\  Q^n &\to Q\\
(q_1,...,q_n)&\mapsto \begin{cases}
\bw\{e(f(p_1,...,p_n)):e(p_1,...,p_n)\geq (q_1,...,q_n)\} \text{ if }\neq\emptyset\\
\top_Q \text{ otherwise.}
\end{cases}
\end{aligned}$
\item[$h_e\colon$] $\begin{aligned}\\ Q &\to \Gamma_e[P^*]\\
q &\mapsto \{p\in P : e(p)\geq q\}
\end{aligned}$
\end{itemize}

\begin{enumerate}
\item 
$\gammap$ is defined to be $h_e\circ \hfp \circ ((h_e)^n)^{-1}$. We have two cases. Suppose first that $C_i \in \Gamma_e[P^*]$ is non-empty for all $i\in\{1,...,n\}$. Then 
\begin{align*}
(C_1,...,C_n)&\mapsto (\bw e[C_1],...,\bw e[C_2])\\
&\mapsto \bw\{e(f(p_1,...,p_n): p_i\in C_i\text{ for all }i)\}=\bw e[f[C_1,...,C_n]]\\
&\mapsto \{p\in P : e(p)\geq \bw e[f[C_1,...,C_n]]\}=\Gamma_e(f[C_1\times\ldots\times C_n]^\uparrow) 
\end{align*}
Alternatively, suppose $C_i=\emptyset$ for some $i$. Then
\begin{align*}
(C_1,...,C_n)&\mapsto (\bw e[C_1],...,\bw e[C_2])\\
&\mapsto \top_Q\\
&\mapsto \Gamma_e(\emptyset) 
\end{align*}
\item This follows from the fact that $\gamma$ is $\iota\circ\gammap$.
\end{enumerate} 
\end{proof}
\end{Lem}

\end{section}

\begin{section}{Preserving inequalities in meet-completions of isotone poset expansions}\label{S;InPreserve}

Given any standard closure operator $\Gamma\colon\Pc^\delta\to\Pc^\delta$ and $n$-ary function $f:P^n\rightarrow P$, there is a map $\gammap\colon(\Gamma[\Pc])^n\to\Gamma[\Pc]$ defined by
\begin{equation*}\gammap(C_1,...,C_n) =\Gamma(f[C_1\times\ldots\times C_n]^\uparrow)\end{equation*}
such that the diagram in figure \ref{F;newCom} commutes. Note that this $\gammap$ is the map that was introduced in lemma \ref{L;mapDef}. The cases in the original definition are redundant so long as we set $f[\emptyset]=\emptyset$.
A special case of this definition is when $f$ is a constant. For this case $\gammap =\Gamma(\set f^\uparrow)=\set f^\uparrow$.

\begin{figure}[ht]
\[\xymatrix{P^n\ar@{->}[r]^{(e_\Gamma)^n}\ar@{->}[d]_{f} & (\Gamma[\Pc])^n\ar@{->}[d]^{\gammap} \\
P\ar@{->}[r]_{e_\Gamma} & \Gamma[\Pc]
}\]\caption{\label{F;newCom}Lifting operations to meet-completions using closure operators}
\end{figure}

\begin{Def}[$\Gamma(\mathcal P^*)$]\label{def:bullet}
Let $\mathcal P=(P, \leq, f_i:i\in I)$  be a poset with isotone operations  $f_i\colon P^{n_i}\rightarrow P$ of arities $n_i$ (for $i\in I$) and let $\Gamma\colon \Pc^\delta\to\Pc^\delta$ be a standard closure operator.  We lift all the operations as in the preceding discussion to define
\begin{equation*}\Gamma(\mathcal P^*)=(\Gamma[P^*], \supseteq, \gammap_i:i\in I)\end{equation*}

\end{Def}
 We note that frequently inequalities that hold with respect to the operations of $\mP$ will fail in this completion. The remainder of this section is devoted to an examination of some conditions which guarantee inequality preservation.
 
\begin{Def}[$\Gamma_\iota$]
Define $\Gamma_\iota$ to be the identity on $\Pc^\delta$.
\end{Def}

\begin{Def}[$\mathscr{L}_{\mP}$]
If $\mP=(P,\leq, f_i:i\in I)$ is an isotone poset expansion then $\mathscr{L}_{\mP}$ is the formal language composed of the standard logical symbols of first order logic along with the signature $\{\leq\}\cup\{f_i:i\in I\}$ that corresponds to the operations of $\mP$ along with the binary relation $\leq$.
\end{Def}
Given an isotone poset expansion $\mP=(P,\leq, f_i:i\in I)$ we shall talk about terms in the language of $\mP$. These are the terms of the language $\mathscr{L}_{\mP}$ constructed as per the usual rules of first order term construction. For example, if $\mP$ has only the single binary operation $f$, then $f(x,y)$ would be a term, as would $f(x,x)$, and $f(f(x,y),z)$ etc. 

\begin{Lem}\label{L;sahl}
Let $\mP=(P,\leq, f_i:i\in I)$ be an isotone poset expansion, let $x_1,...,x_n$ be variables, and let $\phi(x_1,...,x_n)$ be a term in the language of $\mP$. Let $(C_1,...,C_n)\in (\Gamma_\iota(\mathcal P^*))^n$ and consider $\Gamma_\iota(\mathcal P^*)$ as a model for $\mathscr{L}_{\mP}$ by interpreting $f_i$ as $\gamma_i^+$ for all $i\in I$. Suppose we assign $x_i=C_i$ for all $i=1,...,n$. Then the interpretation of $\phi(x_1,...,x_n)$ in $\Gamma_\iota(\mathcal P^*)$ under this assignment is $\{\phi(y_1,...,y_n):(y_1,...,y_n)\in C_1\times...\times C_n\}^\uparrow=\phi[C_1\times...\times C_n]^\uparrow$.
\end{Lem}
\begin{proof}
We induct on the construction of $\phi$. In the base case $\phi=f_i$ for some $i\in I$. In this case the interpretation of $f_i(x_1,...,x_n)$ is just $\gammap_i(C_1,...,C_n)$ and the result follows from the definitions of $\gammap_i$ and $\Gamma_\iota$.

For the inductive step we are interested in the case where $\phi(x_1,...,x_n)=f(\phi_1(\bar{x}_1),...,\phi_n(\bar{x}_n))$, where $f$ is an $n$-ary function from $\mathscr{L}_{\mP}$, $\bar{x}_i$ is a vector of variables from $\{x_1,...,x_n\}$, and $\phi_i(\bar{x}_i))$ is a term in $\mathscr{L}_{\mP}$ for each $i=1,...,n$. To illustrate the proof we will use the special case where $\phi(x_1,x_2,x_3)=f(\phi_1(x_1,x_2),\phi_2(x_3))$ for some $\phi_1$ and $\phi_2$. The general proof is similar but has a tedious notational burden so we choose to omit it. 

Now, by definition \begin{align*}
 \phi[C_1,C_2,C_3]^\uparrow &= \{\phi(x_1,x_2,x_3):x_i\in C_i \text{ for } i=1,2,3\}^\uparrow\\
\tag 1 &= \{f(\phi_1(x_1,x_2),\phi_2(x_3)):x_i\in C_i \text{ for } i=1,2,3\}^\uparrow
\end{align*}

Also, using the inductive hypothesis and the definitions of $\gammap$ and $\Gamma_\iota$ we have that the interpretation of $\phi(x_1,x_2,x_3)$ in $\Gamma_\iota(\mathcal P^*)$ is $f[\phi_1[C_1\times C_2]^\uparrow,\phi_2[C_3]^\uparrow]^\uparrow$.

Now,

\begin{align*}
\tag 2 f[\phi_1[C_1\times C_2]^\uparrow,\phi_2[C_3]^\uparrow]^\uparrow &= \{ f(a,b):a\in\phi_1[C_1\times C_2]^\uparrow,b\in\phi_2[C_3]^\uparrow\}^\uparrow \\
\end{align*}
where $a\geq \phi_1(c_1,c_2)$ for some $(c_1,c_2)\in C_1\times C_2$, and $b\geq \phi_2(c_3)$ for some $c_3\in C_3$. Clearly (1) $\subseteq$ (2), and that (2) $\subseteq$ (1) follows from the fact that $f$ is isotone.

\end{proof}

\begin{Def}[$\mPs$]
Given an isotone poset expansion $\mP=(P,\leq, f_i:i\in I)$ define $\mPs=\Gamma_\iota(\mPs)=(\Pc,\supseteq,\gammap_i:i\in I)$.
\end{Def}

\begin{Prop}\label{P;sahl}
Let $\mP=(P,\leq, f_i:i\in I)$ be an isotone poset expansion, let $\phi(x_1,...,x_n)$ and $\psi(x_1,...,x_n)$ be terms of $\mathscr{L}_{\mP}$. Then \begin{equation*}\mP\models\phi(x_1,...,x_n)\leq\psi(x_1,...,x_n)\iff \mPs\models\phi(x_1,...,x_n)\leq\psi(x_1,...,x_n)\end{equation*}
assuming $\mP$ and $\mPs$ are interpreted as $\mathscr{L}_{\mP}$-structures in the natural way.

\end{Prop}
\begin{proof}
Using lemma \ref{L;sahl} and the definition of $\mPs$
\begin{align*}
&\phantom{\iff}\mPs\models\phi(x_1,...,x_n)\leq\psi(x_1,...,x_n)\\
&\iff \phi[C_1\times...\times C_n]^\uparrow \supseteq \psi[C_1\times...\times C_n]^\uparrow \text{ for all } C_1,...,C_n\in\mPs\\
&\implies \phi[x^\uparrow_1\times...\times x^\uparrow_n]^\uparrow \supseteq \psi[x^\uparrow_1\times...\times x^\uparrow_n]^\uparrow \text{ for all } x_1,...,x_n\in\mP\\
&\iff \mP\models\phi(x_1,...,x_n)\leq\psi(x_1,...,x_n)
\end{align*}
This proves the right to left implication. To prove the other direction let $C_1,...,C_n\in\mPs$. We must show that $\phi[C_1\times...\times C_n]^\uparrow \supseteq \psi[C_1\times...\times C_n]^\uparrow \text{ for all } C_1,...,C_n\in\mPs$ There are two cases. In the trivial case there is $i\in\{1,...,n\}$ with $C_i=\emptyset$. In this case $\psi[C_1\times...\times C_n]^\uparrow=\emptyset=\phi[C_1\times...\times C_n]^\uparrow$. Suppose instead that $C_i\neq\emptyset$ for all $i=1,...,n$ and let $a'\in \psi[C_1\times...\times C_n]^\uparrow$. Then $a'\geq a\in\psi[C_1\times...\times C_n]$ for some $a\in \mP$, and thus $a=\psi(a_1,...,a_n)$ for some $(a_1,...,a_n)\in\mP^n$. Now, by assumption $\phi(a_1,...,a_n)\leq\psi(a_1,...,a_n)$ and so $a'\in \phi[C_1\times...\times C_n]^\uparrow$ and we are done.  
\end{proof}
The following corollary to this result provides a condition on the relationship between a combination of operations of $\mP$ and a standard closure operator $\Gamma$ on $\Pc^\delta$ sufficient to guarantee the preservation of an inequality in the meet-completion induced by $\Gamma$. Somewhat surprisingly it turns out that only the `larger' term is important here.
\begin{Cor}\label{C;sahl}
Let $\mP=(P,\leq, f_i:i\in I)$ be an isotone poset expansion, let $\phi(x_1,...,x_n)$ and $\psi(x_1,...,x_n)$ be terms of $\mathscr{L}_{\mP}$, and let $\Gamma$ be a standard closure operator on $\Pc^\delta$. Define $\Gamma[\mP]=(\Gamma[\Pc],\supseteq,\gammap_i:i\in I)$ as in  definition \ref{def:bullet}. Suppose that $\psi(C_1,...,C_n)=\Gamma(\psi[C_1\times...\times C_n]^\uparrow)$ for all $(C_1,...,C_n)\in\Gamma[\Pc]^n$. Then
\begin{equation*}
\mP\models \phi(x_1,...,x_n)\leq\psi(x_1,...,x_n)\implies\Gamma[\mP]\models\phi(x_1,...,x_n)\leq\psi(x_1,...,x_n).
\end{equation*}
\end{Cor}
\begin{proof}
By proposition \ref{P;sahl} we have $\mPs\models\phi(x_1,...,x_n)\leq\psi(x_1,...,x_n)$, so in particular $\psi[C_1\times...\times C_n]^\uparrow\subseteq \phi[C_1\times...\times C_n]^\uparrow$ for all $(C_1,...,C_n)\in\Gamma[\Pc]^n$. Note that in $\Gamma(\mP)$ we must have $\Gamma(\phi[C_1\times...\times C_n]^\uparrow)\subseteq \phi(C_1,...,C_n)$, and similar for $\psi$. So 
\begin{equation*}\Gamma(\psi[C_1\times...\times C_n]^\uparrow)\subseteq\Gamma(\phi[C_1\times...\times C_n]^\uparrow)\subseteq\phi(C_1,...,C_n)
\end{equation*} 
 and since by assumption $\psi(C_1,...,C_n)=\Gamma(\psi[C_1\times...\times C_n]^\uparrow)$ this gives $\psi(C_1,...,C_n)\subseteq\phi(C_1,...,C_n)$, and thus $\Gamma[\mP]\models\phi(x_1,...,x_n)\leq\psi(x_1,...,x_n)$ as required.
\end{proof}

\begin{Cor}\label{C;sahl2}
With all notation as in corollary \ref{C;sahl} suppose $\mP\models\phi(x_1,...,x_n)=\psi(x_1,...,x_n)$. Then if 
 \begin{enumerate}
\item $\psi(C_1,...,C_n)=\Gamma(\psi[C_1\times...\times C_n]^\uparrow)$, and 
\item $\phi(C_1,...,C_n)=\Gamma(\phi[C_1\times...\times C_n]^\uparrow)$  
\end{enumerate}
in $\Gamma[\mP]$ for all  $(C_1,...,C_n)\in\Gamma[\Pc]^n$, then 
\begin{equation*}
\Gamma[\mP]\models\phi(x_1,...,x_n)=\psi(x_1,...,x_n)
\end{equation*}
\end{Cor}
\begin{proof}
Immediate, from corollary \ref{C;sahl}.\end{proof}
\end{section}

\begin{section}{Ordered domain algebras}\label{S;reps}
The axioms in this section originate with Bredikhin, and the presentation here is that used in \cite{HirMik13}.
\begin{Def}\label{def:rda}
The class $\mathbf{R}(\comp,\dom,\ran,{}^\conv,0,\ide,\leq)$ is defined as the isomorphs of
${\A}=(A,\comp,\dom, \ran, {}^\conv, \emptyset, \id, \subseteq)$ where $A\subseteq\wp(U\times U)$
for some base set $U$ and
\begin{align*}
x\comp y&=\{(u, v)\in U\times U: (u, w)\in x\text{ and }
(w,v)\in y\text{ for some }w\in U\}\\
\dom (x)&=\{(u, u)\in U\times U: (u, v)\in x \text{ for some }v\in U\}\\
\ran (x)&=\{(v, v)\in U\times U:(u, v)\in x \text{ for some }u\in U\}\\
x^\conv &=\{(v, u)\in U\times U:(u, v)\in x\}\\
\ide & =\{(u, v)\in U\times U:u=v\}
\end{align*}
for every $x,y\in A$.
\end{Def}

Let {$\mathbf{Ax}$}\ denote the following formulas:
\begin{description}
\item[Partial order]
$\leq$ is reflexive, transitive and antisymmetric, with lower bound $0$.
\item[Isotonicity and normality]
the operators ${}^\smile, \comp, \dom, \ran$ are isotonic,
e.g.\  $a\leq b\to a\comp c\leq b\comp c$ etc.\  and normal
$0^\smile=0\comp a=a\comp 0=\dom(0)=\ran(0)=0$.
\item[Involuted monoid]
$\comp$ is associative, $\ide$ is left and right identity for $\comp$,
$\ide^\conv=\ide$ and
${}^\smile$ is an involution:
$(a^\conv)^\conv=a,\  (a\comp b)^\conv=b^\conv\comp a^\conv$.
\item[Domain/range axioms]
\mbox{}
\begin{enumerate}
\labitem{(D1)}{D1}$\dom(a)=(\dom(a))^\smile\leq \ide=\dom(\ide)$
\labitem{(D2)}{D2}$\dom(a)\leq a\comp a^\conv$
\labitem{(D3)}{D3}$\dom(a^\conv)=\ran(a)$
\labitem{(D4)}{D4}$\dom(\dom(a))=\dom(a)=\ran(\dom(a))$
\labitem{(D5)}{D5}$\dom(a)\comp a=a$
\labitem{(D6)}{D6}$\dom(a\comp b)=\dom(a\comp\dom(b))$
\labitem{(D7)}{D7}$\dom(\dom(a)\comp\dom(b))=\dom(a)\comp\dom(b)=\dom(b)\comp\dom(a)$
\end{enumerate}
Two consequences of these axioms (use \ref{D6}, \ref{D7} for the first, use \ref{D4}, \ref{D5} for the second) are

\begin{enumerate}
\labitem{(D8)}{D8}$\dom(\dom(a)\comp b)=\dom(a)\comp\dom(b)$
\labitem{(D9)}{D9}$\dom(a)\comp\dom(a)=\dom(a)$
\end{enumerate}
\end{description}
A model of these axioms is called an \emph{ordered domain algebra (ODA)}.

Each of the axioms \ref{D1}--\ref{D8} has a dual axiom,
obtained by swapping domain and range and reversing the order of compositions,
and we denote the dual axiom by a $\partial$ superscript, thus for example,
$\ref{D6}^\partial$ is $\ran(b\comp a)=\ran(\ran(b)\comp a)$.
The dual axioms can be obtained from the axioms above,
using the involution axioms and \ref{D3}.

Another consequence of the ODA axioms is the following lemma, which we shall use later.

\begin{Lem}\label{lem:dab}
Let $\mathcal B$ be any ODA and let $b, c\in \mathcal B$.  Then
\begin{align*}
\dom(b\comp c)\comp b \geq b\comp \dom(c)\\
\intertext{and}
b\comp\ran(c\comp b)\geq \ran(c);b
\end{align*}
\end{Lem}
\begin{proof}
\begin{align*}
\dom(b\comp c)\comp b&=\dom(b\comp\dom(c))\comp b&\text{by \ref{D6}}\\
&\geq\dom(b\comp\dom(c))\comp b\comp\dom(c)&\mbox{\ref{D1}}\\
&=b\comp\dom(c)&\mbox{\ref{D5}}
\end{align*}
The other part is similar.
\end{proof}

\end{section}

\begin{section}{A completion}\label{S;ODAcomp}
\begin{Def}[$\G$]\label{D;closed}
Given an ODA $A$ with underlying poset $P$, define $\G\colon\Pc^\delta\to\Pc^\delta$ by defining the closed sets of $\Pc$ to be those $X\in\Pc$ such that $\{\dom(x)\comp y\comp \ran(z):x,y,z\in X\}^\uparrow=X$.
\end{Def}

\begin{Lem}
$\G$ is a standard closure operator on $\Pc^\delta$.
\end{Lem}
\begin{proof}
Routine.
\end{proof}

\begin{Lem}\label{L;Gconstruction}
Given $X\in\Pc$, if we define 
\begin{align*}&X_0=X, and \\ 
&X_{n+1}=\{\dom(x)\comp y\comp\ran(z):x,y,z\in X_n\}^\uparrow \text{ for all } n\in\omega 
\end{align*}
then $\G(X)=\bigcup_\omega X_n$.
\end{Lem}
\begin{proof}
It's easy to show that $X_n\subseteq X_{n+1}$ for all $n\in\omega$, so given $x,y,z\in \bigcup_\omega X_n$ there is a $k\in\omega$ with $x,y,z\in X_k$. Thus $(\dom(x)\comp y\comp \ran(z))^\uparrow\subseteq X_{k+1}\subseteq \bigcup_\omega X_n$, hence $\bigcup_\omega X_n$ is $\G$-closed. Clearly any closed set containing $X$ must contain $\bigcup_\omega X_n$, so we must have $\G(X)=\bigcup_\omega X_n$ as required.  
\end{proof}

The next definition is a special case of definition~\ref{def:bullet}.
\begin{Def}[\text{$\G[\A]$}]\label{def:closed2}
Given an ODA $\A$ with underlying poset $P$, we define $\G[\A]=(\G[\Pc],\supseteq,\gammap_i:f_i\in\{\comp,\dom,\ran,{}^\conv,0,\id\})$. For clarity we will write $\compC$, $\domC$, $\ranC$, ${}^{\convC}$, $\oC$, $\idC$ for the operations $\gammap_i$.
\end{Def}

\begin{Thm}[Hirsch, Mikulas]\label{thm:rep}
Let $\A$ be an ordered domain algebra.  The map $h:\A\rightarrow \wp(\G[\A]\times\G[\A])$ defined by
\begin{equation*}
(X,Y)\in h(a)\iff X\compC\au\subseteq Y \text{ and } Y\compC (a^\conv)^\uparrow \subseteq X 
\end{equation*}
 is a representation of $\A$ over the base $\G[\A]$.
\end{Thm}
This theorem is proved, with minor notational variations, in the proof of \cite[theorem~2.2]{HirMik13}.

\begin{Lem}\label{L;ops}
Let $\A$ be an ODA with underlying poset $P$ and consider the closure operator $\G$. Then for $f\in\set{\dom, \ran, ^\conv}$ and $C\in\G[\Pc]$ we have $\gammap(C)=f[C]^\uparrow$. Moreover, $\oC=\{0\}^\uparrow$ and $\idC=\{\id\}^\uparrow$.  
\end{Lem}
\begin{proof}
That $\oC$ and $\idC$ are just $\set 0^\uparrow$ and $\set\ide^\uparrow$ is direct from the definition of $\gammap$. For $\dom$ let $C\in\G[\Pc]$ and let $x,y,z\in C$. We are required to show that $\dom[C]^\uparrow$ is $\G$-closed. Now, $\dom(\dom(x))\comp\dom(y)\comp\ran(\dom(z))=\dom(x)\comp\dom(y)\comp\dom(z)=\dom(\dom(x)\comp y)\comp\dom(z)$ by ODA axioms \ref{D4}, \ref{D7}, and \ref{D8}. As $C$ is $\G$-closed we must have  $\dom(x)\comp y\in C$, so we have something of form $\dom(x')\comp\dom(z)$ for $x',z\in C$. Another application of \ref{D8} gives  $\dom(x')\comp\dom(z)=\dom(\dom(x')\comp z)$, and thus as $C$ is closed we have something of form $\dom(y')$ for $y'\in C$, which is in $\dom[C]$. The $\ran$ case is dual, and the $^\conv$ case follows from axiom \ref{D3} and the fact that $^\conv$ is an involution.
\end{proof}


We ask how close $\G[\A]$ is to being an ODA. It turns out that most of the axioms \ref{D1}-\ref{D8} hold (proposition \ref{P;AxHold}), with the exceptions being \ref{D2} and \ref{D6} (Examples \ref{E;2fail} and \ref{E;6fail}), the operations on $\G[\A]$ remain isotone and normal, $\idC$ remains a left and right identity for composition and  $^{\convC}$ is still an involution (Lemma \ref{L;invoetc}). The dramatic deviation is that $\compC$ is not necessarily associative (Example \ref{E;Afail}). The remainder of this section will be taken up with proving the claims in this paragraph.

\begin{Prop}\label{P;AxHold}
Given ODA $\A$, axioms \ref{D1}, \ref{D3}, \ref{D4}, \ref{D5}, and \ref{D7} hold in $\G[\A]$.
\end{Prop}
\begin{proof}
That $\G[\A]\models \{\ref{D1}, \ref{D3}, \ref{D4}\}$ follows easily from corollary \ref{C;sahl} and Lemma \ref{L;ops}. Since $\domC(C_1)\compC\domC(C_2)=\G(\dom[C_1]\comp\dom[C_2]^\uparrow)$ for all $C_1,C_2\in\G[\A]$, by corollary \ref{C;sahl} it is a necessary and sufficient condition for 
$\G[\A]\models \ref{D7}$ that \begin{equation*}\G(\dom[\dom[C_1];\dom[C_2]]^\uparrow)=\dom[\G(\dom[C_1];\dom[C_2])]^\uparrow\end{equation*} for all $C_1,C_2\in\G[\A]$. We shall show that $(\dom[C_1];\dom[C_2])^\uparrow$ is $\G$-closed, as in that case the required equality follows from lemma \ref{L;ops}: Let $x_1,x_2,x_3\in C_1$, and let $y_1,y_2,y_3\in C_2$. Then 
\begin{align*}&\dom(\dom(x_1)\comp\dom(y_1))\comp\dom(x_2)\comp\dom(y_2)\comp \ran(\dom(x_3)\comp\dom(y_3))\\
=&\dom(x_1)\comp\dom(x_2)\comp\dom(x_3)\comp\dom(y_1)\comp\dom(y_2)\comp\dom(y_3)\end{align*}
 by axioms \ref{D4} and \ref{D7}. Since $\dom[C_1]^\uparrow$ and $\dom[C_2]^\uparrow$ are closed by lemma \ref{L;ops} it is easy to show that $\dom(x_1)\comp\dom(x_2)\comp\dom(x_3)\in\dom[C_1]^\uparrow $ and $\dom(y_1)\comp\dom(y_2)\comp\dom(y_3)\in\dom[C_2]^\uparrow$ and thus $\G[\A]\models \ref{D7}$ as required. That $\G[\A]\models \ref{D5}$ follows easily from lemma \ref{L;ops}.
\end{proof}

\begin{Lem}\label{L;invo}
For all $S\in\Pc$, $\G(S)^\conv=\G(S^\conv)$.
\end{Lem}
\begin{proof}
Since $S^\conv\subseteq \G(S)^\conv$ and $\G(S)^\conv$ is $\G$-closed by lemma \ref{L;ops}, $\supseteq$ follows from the isotonicity of closure operators. Define $X_0=S$ and $X_n$ as in lemma \ref{L;Gconstruction} for all $n\in\omega$. Then $X_0^\conv=S^\conv\subseteq\G(S^\conv)$, and for all $k\in\omega$ and every $a\in X_k$ we have  $a\geq b=\dom(b_1)\comp b_2\comp \ran(b_3)$ for some $b_1,b_2,b_3\in X_{k-1}$. So $b^\conv=\dom(b_3^\conv)\comp b_2^\conv\comp \ran(b_1^\conv)$ by involution and axioms \ref{D1} and \ref{D3}, and thus if $X_{k-1}^\conv\subseteq \G(S^\conv)$ then $X_k^\conv\subseteq \G(S^\conv)$. Since $\G(S)^\conv=\bigcup_{n\in\omega}X_n^\conv$ we are done.
\end{proof}

\begin{Lem}\label{L;invoetc}
For all $f\in\{\comp,\dom,\ran,{}^\conv,0,\ide\}$ the extension $\gammap$ is isotone and normal, moreover 
\begin{enumerate}
\item$\idC$ is a left and right identity for $\compC$, and 
\item$^{\convC}$ is an involution.
\end{enumerate}
\end{Lem}
\begin{proof}
Isotonicity of the operations is automatic from the lifting process, and normality follows from the fact that $\oC=\{0\}^\uparrow$. That $\idC$ is a left and right identity for $\compC$  follows easily from the definition of $\compC$ and the fact that $\idC=\{\ide\}^\uparrow$. To see that $(a\comp b)^\conv= b^\conv \comp a^\conv$ holds in $\G[\A]$ define terms $\phi(x,y)=(x\comp y)^\conv$ and $\psi(x,y)= x^\conv \comp y^\conv$ in $\mathscr{L}_{\mP}$. Then using lemma \ref{L;ops} it's easy to see that in $\G[\A]$ we have $\psi(C,D)=\G(\psi[C\times D]^\uparrow)$ for all $C,D\in\G[\A]$, and that $\phi(C,D)=\G(\phi[C\times D]^\uparrow)$ follows from lemma \ref{L;invo}. The result then follows from corollary \ref{C;sahl2}.
\end{proof}

\begin{Lem}\label{lem:union}
Let $X, Y\in\G[\A]$.  If $\domC(X)=\domC(Y)$ and $\ranC(X)=\ranC(Y)$ then $X\cup Y\in\G[\A]$.
\end{Lem}
\begin{proof}
Let $z_1, z_2, z_3\in X\cup Y$.  We are required to prove that \begin{equation*}\dom(z_1);z_2;\ran(z_3)\in X\cup Y\end{equation*}  Without loss of generality, let $z_2\in X$.  Since $\domC(X)=\domC(Y)$ we know that $\dom(z_1)\in \domC(X)$ and similarly $\ran(z_3)\in\ranC(X)$, hence $\dom(z_1);z_2;\ran(z_3) \in X$, by the closure of $X$.
\end{proof}

\begin{Ex}\label{E;2fail}\emph{\ref{D2} can fail in $\G[\A]$}.
Let $\A$ be the full proper ODA over a base of four elements $\{a,b,c,d\}$. Define $x,y\in\A$ by $x=\{(a,b),(c,d)\}$, and $y=\{(a,d),(c,b)\}$. Let $A=\set{x, y}^\uparrow$.   Then $\dom(x)=\dom(y)$ and $\ran(x)=\ran(y)$, and consequently $A$  is $\G$-closed.  We aim to show that $A\compC A^{\convC}\not\subseteq \domC(A)$. For this claim, $x\comp y^\conv\in \G(A\comp A^\conv)^\uparrow$, \/ $x\comp y^\conv=\{(a,c),(c,a)\}$, and $\domC(A)=\dom(x)^\uparrow=\dom(y)^\uparrow=\{(a,a),(c,c)\}^\uparrow$,  so 
$x\comp y^\conv\not\in\domC(A)$ 
hence 
$A\compC A^{\convC} \not\subseteq \domC A$  
and thus
 $\G[\A]\not\models \ref{D2}$. 
\end{Ex}

\begin{Ex}\label{E;6fail}\emph{\ref{D6} can fail in $\G[\A]$}.
Let $\A$ be the full proper ODA over the two element base $\{a,b\}$. Define $x=\{(a,b),(b,a)\}$ and let $\ide=\{(a,a),(b,b)\}$ be the identity as usual. Let $B=\{x,\ide\}^\uparrow$. Then, as $\dom(x)=\dom(\ide)=\ran(x)=\ran(\ide)=\ide$, $B$ is $\G$-closed. Define $A=\set{\{(a, a)\}}^\uparrow$. Then 
\begin{align*}
A\compC  B&=\set{\set{(a,a)}}^\uparrow\compC\set{\ide, x}^\uparrow\\
&=\G(\set{\set{(a, a)}, \set{(a, b)}}^\uparrow)\\
&=\emptyset^\uparrow
\end{align*}
because $\set{(a, b)};\ran\set{(a,a)}=\emptyset$.  Hence $\domC(A\compC B)=\domC(\emptyset^\uparrow)=\emptyset^\uparrow$.
However, $\domC(B)=\idC$ so $\domC(A\compC\domC B)=\domC A=\set{(a, a)}^\uparrow \neq \dom(A\compC B)$ and thus $\G[\A]\not\models \ref{D6}$.
\end{Ex}

\begin{Ex}\label{E;Afail}\emph{Associativity can fail in $\G[\A]$}.
Let $\A$ be the full proper ODA over a base of five elements $\{a,b,c,d,e\}$, let $x=\{(a,a)\}$, let $y=\{(a,b),(c,d)\}$, let $z=\{(a,d),(c,b)\}$, and let $u=\{(b,e),(d,e)\}$. Define $A=\set x^\uparrow,\; B=\{y,z\}^\uparrow$, and $C=\set u^\uparrow$. Then $A$ and $C$ are principal and hence $\G$-closed, and $\dom(z)=\dom(y)$ and $\ran(z)=\ran(y)$ so $B$ is also $\G$-closed. Now, $A\compC B=\G(\{x\comp y, \;  x\comp z\}^\uparrow)=\G(\{\{(a, b)\},\; \{(a,d)\}\}^\uparrow)=\emptyset^\uparrow$, as $\{(a,b)\}\comp \ran(\{(a,d)\})=\emptyset$, so $(A\compC B)\compC C=\emptyset^\uparrow$.  However, $B\compC C=\G\{y\comp u, \; z\comp u\}^\uparrow=\set{\{(a,e),\; (c,e)\}}^\uparrow$, which is principal and hence $\G$-closed. Thus  $A\compC( B\compC C)=\G (\{x\comp y\comp u, \;x\comp z\comp u\}^\uparrow)=\G(\set{\{(a,e)\}}^\uparrow)=\set{\{(a,e)\}}^\uparrow\neq\emptyset^\uparrow$, and so $(A\compC B)\compC C
\neq A\compC( B\compC C)$.
This example also shows that the \emph{weak associativity law}, where associativity is only required for  $A\leq \idC$, can fail in $\G[\A]$.
\end{Ex}

\begin{Prob}
Consider the partial binary relation $*$ on the base of $\G[\A]$ where $B*C$ is only defined (for $B, C\in\G[\A]$)  if $\ranC(B)=\domC(C)$ and then $B*C=B\compC C$.  Is $*$ associative, i.e. is $A*(B*C)=(A*B)*C$ whenever either side is defined?
\end{Prob}
\end{section}

\begin{section}{Conclusions and further work}\label{S:conclusions}
We have shown how to lift the operators of an isotone poset expansion to a meet-completion.  We have identified a family of equations preserved in certain  meet-completions.  For the particular case of ordered domain algebras we have   defined a meet-completion $\G[\A]$ of an ordered domain algebra $\A$ and shown that some, but not all, of the equations defining ordered domain algebras are inherited by $\G[\A]$.  Furthermore we have seen that $\G[\A]$ may be used as the base of a representation of $\A$. 

 It may be of interest to compare all this with similar research in the field of Boolean algebras with operators (BAOs), in particular  relation algebras, where the signature is more expressive.   There are many possible completions of a BAO ranging from the \emph{MacNeille completion} (in some of the references below this is simply called \emph{the} completion) up to the \emph{canonical extension}.   Much work has been done to identify a large class of formulas preserved by completions, e.g. \cite{Sah75,deRVen95,GivVen99}. More generally, canonical extensions, MacNeille completions, and other completions for lattice and poset expansions have received considerable attention in recent years, partly due to their connections to non-classical logics (see e.g. \cite{DGP05,GHV06,GNV05}).

The \emph{canonical extension} of a Boolean algebra with operators can be defined by the second dual of the Boolean algebra, with operators lifted from the algebra \cite{JonTar51}. An equation is \emph{canonical} if it holds in the canonical extension of an algebra whenever it holds in the algebra itself.  Not all equations are canonical, but all the equations defining the class of relation algebras are canonical, \cite{JonTar52} or \cite[theorem~3.16]{HirHod02}.  By a theorem due to J. Monk, the class of all representable relation algebras is a canonical class (an algebra is representable if and only if its canonical extension is, for a proof see \cite{Mad83} or \cite[theorem~3.36]{HirHod02})   but any equational axiomatization of this representation class must involve infinitely many non-canonical equations \cite{HodVen05}.

The \emph{MacNeille completion} of a BAO preserves essentially infinite meets and joins. Again the operators are lifted from the algebra to its completion.  It has been shown that a representable relation algebra can have an unrepresentable MacNeille completion \cite{HirHod02a}, hence some equations fail to be preserved when passing to the MacNeille completion.

Bearing this in mind there are two primary directions the work here can be extended. First, by building a deeper understanding of the preservation of inequalities by meet-completions, in the spirit of \cite{Sah75} and in particular its numerous algebraic descendents (e.g. \cite{GNV05,GivVen99,Gold89}). Section \ref{S;InPreserve} makes a small step in this direction, but it seems likely that the results here could be refined and extended considerably. Second, the inspiration for this paper was the implicit appearance of the meet-completion structure described in section \ref{S;ODAcomp} in the results of \cite{HirMik13}. It remains to be seen whether this is an isolated event or whether the role of the completion in the representation process hints at some deeper structure which could possibly be used for further representation results. With these thoughts in mind we propose some problems we believe are of particular interest:

\begin{Prob} Let $\A$ be an ODA and let $\Gamma$ be a standard closure operator defining a completion of $\A$.
Under what conditions is the completion $\G[\A]$ an ODA, or at least when is the completion associative? More generally, under what conditions is $\Gamma[\A]$ associative?
\end{Prob}
\begin{Prob}
Is it possible to generalize the definition of a representation of an ordered domain algebra in such a way that the completion $\G[\A]$ of an ordered domain algebra does possess a weak representation.  C.f. For relation algebras we can generalize the notion of a representation to a \emph{relativized representation} where all operators are restricted to some reflexive and symmetric (but not necessarily transitive) maximal relation.  When evaluated in a relativized representation, composition need not  be associative. 
A \emph{weakly associative algebra} is a relation-type algebra obeying all the relation algebra axioms except perhaps associativity and satisfying the weak associativity axiom $(x;1);1=x;(1;1)$ instead.  Maddux proved that the class of relation-type algebras isomorphic to algebras of binary relations with relativized operators is  the class of weakly associative algebras \cite{Mad82}.

\end{Prob}

\begin{Prob}
All of the ODA axioms except associativity, \ref{D2} and \ref{D6} are valid over $\G(ODA)=\set{\G[\A]:\A \in {ODA}}$.     Can this set of axioms be extended to a (finite) set of formulas (or equations)  so as to define the closure under isomorphism of $\G(ODA)$?
\end{Prob}
\end{section}

\end{document}